\documentclass[12pt]{article}
\pdfoutput=1

\usepackage[pdftex]{graphicx}

\usepackage{amsmath,amssymb,amsthm, amsfonts}
\usepackage{hyperref}
\usepackage{graphicx}
\usepackage{url}
\usepackage[mathlines]{lineno}
\usepackage{dsfont} 
\usepackage{color}
\definecolor{red}{rgb}{1,0,0}

\definecolor{blue}{rgb}{0,0,1}

\definecolor{green}{rgb}{0,.6,0}

\newtheorem{thm}{Theorem}[section]

\newtheorem{lem}[thm]{Lemma}
\newtheorem{prop}[thm]{Proposition}

\theoremstyle{definition}

\theoremstyle{definition}
\newtheorem{defn}[thm]{Definition}

\theoremstyle{definition}
\newtheorem{ex}[thm]{Example}

\newcommand{\la}{\langle}
\newcommand{\ra}{\rangle}

\def\({\left(}
\def\){\right)}
\title{The Polychromatic Number of Small Subsets of the Integers
Modulo $n$}
\author{Emelie Curl, John Goldwasser, Joe Sampson, Michael Young}

\begin{document}
\maketitle

\begin{abstract} 
\noindent 
If $S$ is a subset of an abelian group $G$, the {\em polychromatic number} of $S$ in $G$ is the largest integer $k$ so that there is a $k-$coloring of the elements of $G$ such that every translate of $S$ in $G$ gets all $k$ colors.  We determine the polychromatic number of all sets of size 2 or 3 in the group of integers mod n.   
\end{abstract}



\section{Introduction}\label{sintro}  
\indent Throughout this paper $G$ will denote an arbitrary abelian group. Given $S \subseteq G$, $a \in G$, $a + S = \{a+s| s \in S\}$. Any set of the form $a + S$ is called a {\em translate } of $S$. A $k$-coloring of the elements of $G$ is {\em $S-$polychromatic} if every translate of $S$ contains an element of each of the $k$ colors. The {\em polychromatic number} of $S$ in $G$, denoted $p_G(S)$, is the largest number of colors such that there exists an $S-$polychromatic coloring of $G$. The notation $p(S)$ is used when $G$ is the set of integers, $\mathbb{Z}$, and $p_n(S)$ is used when $G = \mathbb{Z}_n$, the group of integers $\bmod \hspace{1mm}n$.  In this paper, $p_n(S)$ is determined for all $n \ge 3$ and $|S| = 2$ or $3$.  The techniques used may be useful in determining $p_n(S)$ for larger sets $S$ and for other coloring problems.

The notions of polychromatic colorings and polychromatic number for sets in abelian groups can be extended.  If $G$ is any structure and $H$ is a family of substructures then a $k-$coloring of $G$ is  $H$-polychromatic if every member of $H$ gets all $k$ colors, and the polychromatic number $p_G(H)$ of $H$ in $G$ is the largest $k$ such that there is an $H$-polychromatic coloring with k colors.  In this paper, $G$ is $\mathbb{Z}_n$ and $H$ is the family of all translates of a subset $S$.  Alon et.al. \cite{AKS07}, Bialostocki \cite{Bia}, Offner \cite{Off08}, and Goldwasser et.al. \cite{GLMOTY16} considered the case when $G$ is an $n$-cube and $H$ is the family of all sub-$d$-cubes for some fixed $d \le n$.  Axenovich et. al. \cite{Axe2018} considered the case where $G$ is the complete graph on $n$ vertices and $H$ is the family of all perfect matchings or Hamiltonian cycles or 2-factors.  

If $S$ and $T$ are subsets of an abelian group $G$, we say $T$ is a blocking set for $S$ if $G \setminus T$ contains no translate of $S$.  Blocking sets are of interest in extremal combinatorics, because if $T$ is a minimum size blocking set for $S$ then $G \setminus T$ is a maximum size subset of $G$ with no translate of $S$, so is the solution to a Tur\'{a}n-type problem. It is well known (\cite{Axe2019},\cite{St86}) that $T$ is a complement set for $S$ if and only if $-T$ is a blocking set for $S$. Clearly each color class in an $S$-polychromatic coloring is a blocking set for $S$.
 
In \cite{Axe2019}, Axenovich et. al. considered the situation when $G$ is the group of integers and $H$ is the family of all translations of a set $S$ of 4 integers. They showed that the polychromatic number of any set $S$ of 4 integers in $\mathbb{Z}$ is at least 3, by finding a particular value of $n$ such that $3 \le p_n(S)$.  That implies that any set $S$ of size 4 has a blocking set in $\mathbb{Z}$ of density at most $1/3$, proving a conjecture of Newman about densities of complement sets.  

Whereas in \cite{Axe2019} it was shown that for each set $S$ of integers of size 4, there exists an integer $n$ such that $3 \le p_n(S)$, such an inequality does not hold for all $S$ and $n$.  For example, if $S = \{0,1,3,6\}$ and $n=11$, then $p_n(S)=2$.  It would be difficult to determine $p_n(S)$ for all values of $n$ and all sets $S$ of size 4, but in this paper these values are determined for all sets $S$ of size 3.

\begin{ex}\label{expolynum3}  Let $S = \{0,a,b\}$ be a subset of $\mathbb{Z}_n$ where $n$ is divisible by 3, $a \equiv 1 \hspace{1mm} (\bmod \hspace{1mm} 3)$, and $b \equiv 2 \hspace{1mm} (\bmod \hspace{1mm} 3)$.  Then $p_n(S) = 3$ as the coloring $RBYRBY\ldots$ is obviously $S-$polychromatic.
\end{ex}

\begin{ex}\label{exfanoplane} If $S = \{0,1,3\}$ and $n = 7$ then $p_n(S) = 1$.\\ 
\begin{figure}[h]
\begin{center}
\includegraphics[width=0.7\textwidth]{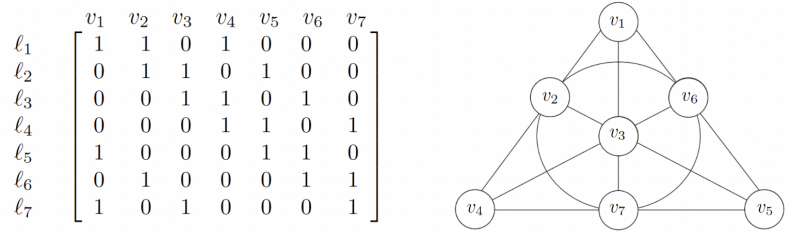}\label{fanoandmatfig}
\caption{Fano plane and an incidence matrix}
\end{center}
\end{figure}

Consider the above figure and note that the $7 \times 7$ circulant matrix is an incidence matrix for the Fano plane.  It is well known (and it is easy to check) that in any $2-$coloring of the vertices of the Fano plane there is a monochromatic edge, which implies there is no $S-$polychromatic $2-$coloring, so $p_7(S)=1$.
\end{ex}

The main result of this paper is that examples \ref{expolynum3} and \ref{fanoandmatfig} are essentially the only examples of sets $S$ of size three such that $p_n(S)$ is not equal to 2.

\section{Simplifying assumptions and the main theorem}

The polychromatic number of a set $S$ in $\mathbb{Z}_n$ is unchanged under certain operations involving translation, multiplication, and scaling.  If $|S|=3$ we can use those operations to convert a set $S$ to a set $S'$ which has the same polychromatic number, and has one of two specific forms. 

\begin{lem}\label{polyequiv}
If $1 \le d,t,n \in \mathbb{Z}$, $S=\{a_1, a_2, \ldots a_t \} \subseteq \mathbb{Z}_n$,  and $S'=\{da_1, da_2, \ldots da_t \}$, then $p_{dn}(S') = p_{n}(S)$.
\end{lem}

\begin{proof}
Any $S-$polychormatic coloring of $\mathbb{Z}_n$ can clearly be copied on the subgroup $\langle d \rangle$ of $\mathbb{Z}_{dn}$, and then duplicated on all the cosets of $\langle d \rangle$, to get an $S'-$polychromatic coloring of $\mathbb{Z}_{dn}$.  Going the other way, in any $S'-$polychromatic coloring of $\mathbb{Z}_{dn}$, the restricted coloring on $\langle d \rangle$ can be copied on $\mathbb{Z}_n$ to get an $S-$polychromatic coloring.
\end{proof}

Hence we can simply divide out a common factor of $n$ and the elements of $S$ without changing the polychromatic number.  Since we can also take any translation of $S$ without changing the polychromatic number, from now on we will assume that every set $S$ of size 3 in $\mathbb{Z}_n$ has the form $S = \{0, a, b\}$ where $\gcd(a,b,n) = 1$.  

\begin{lem}\label{polyequiv2}
Let $1 \le d,t,n \in \mathbb{Z}$ such that $d < n$ and $\gcd(d,n) = 1$. If $S'=\{da_1, da_2, \ldots da_t \}$ and $S=\{a_1, a_2, \ldots a_t \}$, then $p_{n}(S) = p_{n}(S')$.
\end{lem}

\begin{proof}
If $\chi'$ is $S'-$polychromatic, the coloring $\chi$ defined by $\chi(y) = \chi'(dy)$ is clearly $S-$polychromatic.  This argument can be reversed since $d$ is invertible in $\mathbb{Z}_n$.
\end{proof}

\begin{defn}\label{equivalentdef}
If $S = \{a_1, a_2,\ldots,a_t\} \subseteq  \mathbb{Z}_n$ and $S' = \{da_1+c, da_2+c,\ldots,da_t+c\}$, where $c,d \in \mathbb{Z}_n$ and $\gcd(d,n)=1$, then we say that $S$ and $S'$ are {\em equivalent} sets in $\mathbb{Z}_n$. 
\end{defn}

Thus, Lemma \ref{polyequiv2} says that equivalent sets in $\mathbb{Z}_n$ have the same polychromatic number.

\begin{lem}\label{ellsmall}
For all $b \in \mathbb{Z}_n$ with $3 \le n$ there exists $b' \in \mathbb{Z}_n$ so that $b' \leq \lceil \frac{n}{2} \rceil$ and $p(\{0,1,b\}) = p(\{0,1,b'\})$.
\end{lem}
\begin{proof}
Since, $n-1$ is always relatively prime to $n$ for $3 \le n$, $p_n(S) = p_n(-S)$ for all $S \subseteq \mathbb{Z}_n$ by Lemma \ref{polyequiv2}.
If $\lceil \frac{n}{2} \rceil < b$, then let $b' = n - b + 1 \leq \frac{n}{2}$. 
Therefore, $p(\{0,1,b\}) = p(\{-1, 0, -b\}) = p( \{0,1,  - b + 1\}) = p(\{0,1, n - b + 1\})$. 
\end{proof}

\begin{prop}\label{reduce0abto01b}
 Let $S = \{0, a, b\} \subseteq \mathbb{Z}_n$ where $\gcd(a,b,n)=1$.  Then at least one of the following occurs.
 \vspace{-3mm}
 \begin{itemize}
\item[i.]  $S$ is equivalent to a set $S' = \{0, 1, b'\}$ where $b' \leq \lceil \frac{n}{2}\rceil$.
\vspace{-3mm}
\item[ii.]  $\gcd(a,n)\neq 1$, $\gcd(b,n)\neq 1$, $a\notin \langle b \rangle$ and $b \notin \langle a \rangle$.
\end{itemize}
\end{prop} 

\begin{proof}
 If $\gcd(a,n)=1$ then $a$ is invertible in $\mathbb{Z}_n$, so $S$ is equivalent to a set $\{0,1,c\}$, for some $c$ ($d=a^{-1}$ in Definition \ref{equivalentdef}), and then to $S'$ by Lemma \ref{ellsmall}.  Similarly if $\gcd(b,n)=1$.  Now suppose neither $\gcd(a,n)$ nor $\gcd(b,n)$ is equal to 1.  If $b$ is a multiple of $a$ then, since $\gcd(a,b,n)=1$, $\gcd(a,n)$ must equal 1, a contradiction, so $b$ is not a multiple of $a$.  Similarly, $a$ is not a multiple of $b$.
\end{proof}

We remark that if Case $ii$ occurs and $\gcd(b-a, n)=1$, then Case $i$ also occurs. However, in our proof we just need that at least one of them occurs.  We will treat Case $i$ in Section 5 and Case $ii$ in Section 6. The following theorem is the main result of this paper.

\begin{thm}\label{reducedsubsetsize3}
Let $S=\{0, a, b\} \subseteq \mathbb{Z}_n$ and $\gcd(a,b,n) = 1$, then
\[p_{n}(S) = \begin{cases}
3 &\text{ if } 3|n \text{ and }a \text{ and }b \text{ are in different nonzero}\bmod 3 \text{ congruence classes }\\
1 &\text{ if } n=7 \text{ and } \{0,a,b\} \text{ is equivalent to }\{0,1,3\}\\
2 &\text{ otherwise}.
\end{cases}\]
\end{thm}

If we do not make the assumption that $\gcd(a,b,n)=1$, then we get the following theorem, which is clearly equivalent to Theorem \ref{reducedsubsetsize3}:

\begin{thm}\label{subsetsize3}
If $3 \le n$, $a,b \in \mathbb{Z}_n$, and $a\neq b$, then
\[p_{n}(\{0,a,b\}) = \begin{cases}
3 &\text{ if } n \equiv 0 \bmod 3^{j+1}, a = 3^j m_a, b = 3^j m_b,\\
   & m_a,m_b \not\equiv 0 \bmod 3,\text{and } m_a + m_b \equiv 0 \bmod 3\\
1 &\text{ if } n \equiv 0 \bmod 7, |\langle a \rangle| = 7, \text{and } b=3a \text{ or } 5a\\
2 &\text{ otherwise}.
\end{cases}\]
\end{thm}

\section{Sets of size 2}

For the following proposition we assume without loss of generality that $0$ is in the chosen subset of $\mathbb{Z}_n$.
\begin{prop}\label{subsetsize2}
If  $S = \{0,b\} \subseteq \mathbb{Z}_n$ where $\gcd(b,n) = 1$ then
\[p_{n}(S) = \begin{cases}
1 & \text{ if }|\la b \ra| \text{ is odd}\\
2 & \text{ if }|\la b \ra| \text{ is even}.
\end{cases}\]
\end{prop}

\begin{proof}

Clearly there will be an $S-$polychromatic $2$-coloring of the multiples of $b$ if and only if  $|\la b \ra|$ is even.
\end{proof}

\section{Sets that tile}

 Given a set $S \subseteq G$ where $G$ is an abelian group,  a set $T \subseteq G$ is a {\em complement set} for $S$ if $S + T = G$. $S$ \emph{tiles} $G$ by translation if $T$ is a complement set for $S$ and if $s_1, s_2 \in S$, $t_1, t_2 \in T$, and $s_1 + t_1 = s_2 + t_2$ implies $s_1 = s_2$ and $t_1 = t_2$. The notation $S \oplus T$ is used when $S$ tiles $G$ by translation. Without loss of generality, $0 \in S, T$ for all of the following arguments.

Newman~\cite{New77} proved necessary and sufficient conditions for a finite set $S$ to tile $\mathbb{Z}$ if $|S|$ is a power of a prime. 

\begin{thm}\cite{New77}\label{newtile77}
Let $S=\{s_1, \ldots, s_k\}$ be distinct integers with $|S|=p^\alpha$ where $p$ is prime and $\alpha$ is a positive integer.  For $1\le i<j\le k$ let $p^{e_{ij}}$ be the highest power of $p$ that divides $s_i-s_j$.  Then $S$ tiles $\mathbb{Z}$ if and only if $|\{e_{ij}: 1 \le i <j \le k\}| \le \alpha$.
\end{thm} 

The characterization of sets $S$ of size $3$ such that $p_n(S)=3$ (Theorem \ref{reducedsubsetsize3} and Proposition \ref{P3iffcond3}) follows immediately from Newman's theorem (Theorem \ref{newtile77}). When commenting on this theorem in \cite{New77} Newman says: ``Surely the special case [when $|S|=3$] deserves to have a completely trivial proof - but we have not been able to find one."

If there is an $S-$polychromatic $k-$coloring of $\mathbb{Z}_n$, then clearly there is an $S-$polychromatic $k-$coloring of $\mathbb{Z}$ with period $n$.  If there is an $S-$polychromatic $k-$coloring of $\mathbb{Z}$ for a finite set $S$, then there is an $S-$polychromatic $k-$coloring of $\mathbb{Z}_n$ for some $n$.  To see this, let $d$ equal the largest difference between two elements in $S$.  If $\chi$ is an $S-$polychromatic $k-$coloring of $\mathbb{Z}$, there are only $k^{(d+1)}$ possibilities for the coloring on $d+1$ consecutive integers, so two such strings must be identical.  If $n$ is the difference between the first integers in these two strings, then we can ``wrap around" the coloring $\chi$ to get an $S-$polychromatic $k-$coloring of $\mathbb{Z}_n$.

Suppose $S=\{0,a,b\}$ and $\chi$ is an $S-$polychromatic $3-$coloring of $\mathbb{Z}$.  By the above remark there exists an $S-$polychromatic $3-$coloring of $\mathbb{Z}_n$ for some $n$.  By Proposition \ref{P3iffcond3}, $a$ and $b$ are in different nonzero $\bmod \hspace{1mm}3$ congruence classes, which fulfills Newman's wish to have a simple proof of his theorem for the special case when $|S|=3$.

Later Coven and Meyerowitz~\cite{CM99} gave necessary and sufficient conditions for $S$ to tile $\mathbb{Z}$ when $|S|= p_1^{\alpha_1}p_2^{\alpha_2}$, where $p_1$ and $p_2$ are primes. The following characterization of tiling by translation in an abelian group was obtained in \cite{Axe2019}.

\begin{thm}\cite{Axe2019} \label{polyifftiles}
Let $G$ be an abelian group and $S$ a finite subset of $G$. $S$ tiles $G$ by translation if and only if $p(S) = |S|$. Moreover, if $\chi$ is an $S-$polychromatic coloring of $G$ with $|S|$ colors and $T$ is a color class of $\chi$, then $S \oplus T = G$.
\end{thm}

\begin{lem} \label{xinTthenxabinT}
Suppose $S = \{0,a,b\}$ where $\gcd(a,b,n) = 1$, $S \oplus T = \mathbb{Z}_n$ and $0 \in T$. If $x \in T$, then $x + \la a + b \ra \subseteq T$.
\end{lem}
\begin{proof}
Note that because $S \oplus T = \mathbb{Z}_n$, every element of $\mathbb{Z}_n$ belongs to exactly one of the sets $T$, $a + T$, $b + T$.

Suppose $x \in T$. If $x + a + b \in b + T$, then $x + a \in T$. However, $x + a \in a + T$. If $x + a + b \in a + T$, then $x + b \in T$. However, $x + b \in b + T$. Hence $x + a + b \in T$ and, repeating the argument, $x + \la a + b \ra \subseteq T$.

\end{proof}

\begin{prop} \label{P3iffcond3}
Let $S = \{0, a, b\} \subseteq \mathbb{Z}_n$ where $\gcd(a,b,n)=1$.  Then $p_n(S)=3$ if and only if $3|n$ and $a$ and $b$ are in different nonzero $\bmod \hspace{1mm} 3$ congruence classes .
\end{prop}

\begin{proof} If $3|n$ and $a$ and $b$ are in different nonzero $\bmod \hspace{1mm} 3$ congruence classes then clearly the alternating coloring $RBYRBY\ldots$ is polychromatic, so $p_n(S)=3$.  Conversely, suppose $p_n(S)=3$.  Hence, by Theorem \ref{polyifftiles}, $S$ tiles $\mathbb{Z}_n$.  

Let $T \subseteq \mathbb{Z}_n$ such that $\mathbb{Z}_n = \{0,a,b\} \oplus T$ and $0 \in T \subseteq \mathbb{Z}_n$. Therefore, $n = 3|T|$ which implies $n \equiv 0 \bmod 3$. By Lemma \ref{xinTthenxabinT}, for any $x \in T$, the coset $x + \langle a+b \rangle$ is a subset of $T$, so $T$ is the disjoint untion of cosets of $\la a + b \ra$. Therefore, there is some integer $q$ such that $q |\langle a+b \rangle| = |T| = \frac{n}{3}$. Also, $|\langle a + b \rangle| = \frac{n}{gcd(a+b,n)}$. Thus, $q \frac{n}{gcd(a+b,n)} = \frac{n}{3}$, which implies $3q = gcd(a+b, n)$. Hence, $3|(a+b)$. Since 3 cannot divide both $a$ and $b$, it follows that $a$ and $b$ are in different nonzero mod 3 congruence classes.        
\end{proof}

\section{Subsets of the form $\{0,1,b\}$}

As shown in Proposition \ref{reduce0abto01b}, every set $S$ of size 3 is equivalent to a set $S'$ with two possible forms.  In this section we will consider case $i$ of Proposition \ref{reduce0abto01b}, that $S'$ contains 0 and 1.

\begin{lem}\label{color013}
If $n$ is odd, $5 \le n$, and $n \neq 7$, then there exists a $\{0,1,3\}-$polychromatic coloring of $\mathbb{Z}_n$ with two colors.
\end{lem}

\begin{proof}
It is easy to check that each integer greater than 3, except 7, is the sum of an even number of 2's and 3's.  We color $\mathbb{Z}_n$ by alternating colors of strings of 2 or 3 consecutive elements with the same color.  Of course there must be an even number of strings.  For example, 9=2+2+2+3, so the coloring would be $RRBBRRBBB$; 11=2+3+3+3, so the coloring would be $RRBBBRRRBBB$.  Clearly any translate of S hits two consecutive strings, so gets both colors.       
\end{proof}

As will be seen in the proof of Theorem \ref{theconclof01ell}, it is easy to show that $p_n(\{0, 1, b\}) \geq 2$ if $b$ or $n$ is even.  The following lemma takes care of the more difficult case.

\begin{lem}\label{firstcopremain01ell}
Let $9 \le n$, $b$ and $n$ both be odd, and $S = \{0,1, b\} \subset \mathbb{Z}_n$. There exists an $S-$polychromatic coloring of $\mathbb{Z}_n$ with two colors.
\end{lem}

\begin{proof} 
It can be assumed that $5 \le b \leq \lceil \frac{n}{2} \rceil$, by Lemma \ref{ellsmall} and \ref{color013}, and $n = m(b - 2) + r$, with $0 \le r \le b-3$. Since $2(\lceil \frac{n}{2} \rceil -2)+r \le n-3+r \le n$, $m$ is at least 2.

Let $x \in \mathbb{Z}_n$ and $x \equiv y \bmod (b -2)$ such that $0 \le y \le b - 3$. If $r=0$, then define $\chi_0: \mathbb{Z}_n \to \{R, B\}$ such that 
\[\chi_0(x) = \begin{cases}
R &\text{ if } y = 0\\
R &\text{ if }  y \text{ is odd}\\
B &\text{ if } y \text{ is even} \text{ and } y > 0 .
\end{cases}\]

\noindent If $\chi_0(x) = \chi_0(x+1)$, then $x \equiv 0 \bmod b -2$ and $\chi_0(x+b) = B$. This means that $\chi_0(x) \neq \chi_0(x+1)$ or $\chi_0(x) \neq \chi_0(x+b)$. Therefore, every translate of $S = \{x, x+1, x+b\}$ will contain two colors under $\chi_0$.

Throughout the remainder of the proof each of the colorings that are constructed will use $\chi_0$ to assign colors to at least the first $(m-1)(b-2)$ elements of $\mathbb{Z}_n$. 

If $r=1$, then define $\chi_1: \mathbb{Z}_n \to \{R, B\}$ such that
\[\chi_1(x) = \begin{cases}
\chi_0(x) &\text{ if } x \le n-b \\
R &\text{ if } x  \text{ is even}\text{ and } n - b < x < n-1 \\

B &\text{ if } x  \text{ is odd} \text{ and } n - b < x < n-1 \\
B &\text{ if } x  = n-1.
\end{cases}\]

\noindent 
In $\chi_1$, the two translates that are not colored completely by $\chi_0$ and don't have $\chi_1(x) \neq \chi_1(x+1)$ are $\{n - b, n - b + 1, 0\}$ and $\{n - 2, n - 1, b - 2\}$. In both cases, the nonconsecutive element of the translate is the other color.

If $r=2$, then define $\chi_2: \mathbb{Z}_n \to \{R, B\}$ such that
\[\chi_2(x) = \begin{cases}
\chi_0(x) &\text{ if } x \le n-b-1 \\
R &\text{ if } x  = n-b \\
B &\text{ if } x  = n- b +1 \\
R &\text{ if } x  \text{ is odd}\text{ and } n - b +1 < x\\
B &\text{ if } x  \text{ is even} \text{ and } n - b +1 < x.
\end{cases}\]

\noindent 
In $\chi_2$, the only translate that is not colored completely by $\chi_0$ and doesn't have $\chi_2(x) \neq \chi_2(x+1)$ is $\{n - b + 1, n - b + 2, 1\}$; however, $\chi_2(n - b + 1) \neq \chi_2(1)$.

If $r=3$, then define $\chi_3: \mathbb{Z}_n \to \{R, B\}$ such that
\[\chi_3(x) = \begin{cases}
\chi_0(x) &\text{ if } x \le n-b-2 \\
R &\text{ if } x  = n-b -1\\
R &\text{ if } x  \text{ is even}\text{ and } n - b -1 < x < n-1\\
B &\text{ if } x  \text{ is odd} \text{ and } n - b -1 < x < n-1\\
B &\text{ if } x  = n- 1.
\end{cases}\]

\noindent 
In $\chi_3$, the two translates that are not colored completely by $\chi_0$ and don't have $\chi_3(x) \neq \chi_3(x+1)$ are $\{n - b - 1, n - b, n - 1\}$ and $\{n - 2, n - 1, b - 2\}$. In both cases, the nonconsecutive element of the translate is the other color.

Assume $4 \le r$. An $S$-polychromatic coloring, $\chi_4:  \mathbb{Z}_n \to \{R, B\}$, will be constructed. Define $\chi_4(x) = \chi_0(x)$ for $x \le n - b - r +4$, $\chi_4(n-r+2) = B$ and $\chi_4(n-1) = B$. So each translate with $0 \le x \le n - b - r+3$ contains both colors. The two translates with $n-2 \le x$ also contain both colors even though $\chi_4(n-2)$ has not been defined unless $r=4$. This means there is an \emph{option} for assigning a color to $n-2$. Therefore, $\chi_4(x)$ can be defined, and will be defined, such that for $n- b+2 \le x \le n-2$ the assigned colors alternate while keeping  $\chi_4(n-r+2) = B$. This means that each translate with $n - b +2 \le x$ contains both colors.

If $r=4$, then $n- b = n - b -r +4$ has already been assigned the color $B$ and the translate $\{n-b, n-b+1, 0\}$ contains both colors. By defining $\chi_4(n-b+1)$ to be $B$, the translate $\{n-b+1, n-b+2, 1\}$ contains both colors and $\chi_4$ is an $S$-polychromatic coloring.

For $r \neq 4$, consider the translates $\{n - b, n - b + 1, 0 \}$ and $\{n - b + 1, n - b + 2, 1\}$. If $\chi_4(n - b + 2) = B$, then $n - b + 1$ has an option since $\chi_4(1) = R$. If $\chi_4(n - b + 2) = R$, then $\chi_4(n - b +1)$ must be defined as $B$ and $n- b$ has an option  since $\chi_4(0) = R$. Therefore, there will be an option for assigning a color to $n-b$ or $n-b +1$. This allows for $\chi_4(x)$ to be defined for $n-b-r+5 \le x \le n-b+1$ such that the colors alernate while keeping  $\chi_4(n-b-r+4)\neq \chi_4(n-b-r+5)$. Thus, $\chi_4$ is an $S$-polychromatic coloring.
\end{proof}

\begin{thm}\label{theconclof01ell}
Let $3 \le n$ and $S = \{0, 1, b \} \subseteq \mathbb{Z}_n$. If $n \neq 7$ or $b \neq 3$ or $5$, there is an $S$-polychromatic coloring of $\mathbb{Z}_n$ with two colors.
\end{thm}

\begin{proof}
If $n$ is even then alternating colors $RBRB\ldots$ is clearly an $S-$polychromatic coloring.  If $n$ is odd and $b$ is even then the coloring $RRBRBRBR\ldots$ which has one repeated color, and otherwise alternates colors, is $S-$polychromatic.  If $n$ and $b$ are both odd then an $S-$polychromatic $2-$coloring exists by Lemmas \ref{color013} and \ref{firstcopremain01ell}, except in the exceptional case when $n=7$.     
\end{proof}

\section{Subsets not equivalent to $\{0,1,b\}$}

\noindent Consider the $s \times t$ matrix
\[M =
\begin{bmatrix}
x_{00} & x_{01} & \ldots & x_{0(t-1)}\\
x_{10} & x_{11} & \ldots & x_{1(t-1)}\\
\vdots & \vdots & \ddots & \vdots\\
x_{(s-1)0} & x_{(s-1)1} & \ldots & x_{(s-1)(t-1)}.
\end{bmatrix}\]

\noindent An {\em ell - tile} of $M$ is a subset of entries of $M$ consisting of entries of a $2 \times 2$ submatrix without the lower right entry:
\[\begin{array}[t]{cc}
               \begin{array}{|c|c|} 
\hline
x_{ij} & x_{i (j+1)} \\ 
\hline
x_{(i+1)j} &  \\
\hline
\end{array}.
\end{array}\]

The indices are read $\bmod \hspace{1mm} s$ and $\bmod \hspace{1mm} t$, so ell-tiles are allowed to `wrap around' ($i=s-1$ or $j=t-1$).  An {\em ell-tile $2-$coloring} of $M$ is a coloring of the entries of $M$ with two colors such that both colors appear in every ell-tile of $M$.

\begin{lem} \label{matrixcoloring}
If $2 \le s,t$, then every $s \times t$ matrix has an ell - tile $2$-coloring.
\end{lem}

\begin{proof}
If $s$ is even, then define $\chi$ such that
\[\chi(x_{ij}) = \begin{cases}
R &\text{ if } i \equiv 0 \bmod 2\\
B &\text{ if } i \equiv 1 \bmod 2.
\end{cases}\] 
\noindent
Also, a similar coloring that alternates the colors of the columns works when $t$ is even.

If $s$ and $t$ are both odd, then define $\chi$ such that
\[\chi(x_{ij}) = \begin{cases}
R &\text{if } i \equiv j \bmod 2 \text{ and } (i,j) \neq (0, t-1), (s-1, 0)\\
B &\text{ otherwise.}
\end{cases}\] 
If $s$ and $t$ are both odd, then a ``checker-board"€ coloring would assign the same color, say $R$, to all four corner entries, and the ell-tile with entries $x_{s-1,t-1}$, $x_{0,t-1}$, and $x_{s-1,0}$ would be monochromatic.  The coloring $\chi$ avoids this problem by changing the color of entries $x_{0,t-1}$ and $x_{s-1,0}$ from $R$ to $B$, without creating any other monochromatic ell-tiles (just changing the color of one of them would suffice as well).  
\end{proof}

The goal now is to create matrices with elements from $\mathbb{Z}_n$ such that all of the translates of $S$ correspond to ell-tiles. The matrices then can be colored by using Lemma \ref{matrixcoloring}, which will create $S-$polychromatic colorings.

\begin{lem}\label{2poly}
Let $S = \{0,a,b\} \subseteq \mathbb{Z}_n$, where $gcd(a,b,n) = 1$ but $\gcd(a,n)$ and $\gcd(b,n)$ are both greater than 1. Then $p_n(S) \geq 2$.
\end{lem}

\begin{proof}
If $n$ is even then either $a$ or $b$ is odd, so the alternating coloring $RBRBRB\ldots$ is polychromatic, so we can assume $n$ is odd.  Let $s=\gcd(a,n)$, $t=\gcd(b,n)$, and $M = [m_{ij}]$ be the $\frac{n}{s} \times \frac{n}{t}$ matrix with entries in $\mathbb{Z}_n$ where $m_{ij} = ai + bj$, $0 \leq i \leq \frac{n}{s} - 1$, $0 \leq j \leq \frac{n}{t} - 1$.  Note that $|\la a \ra| = \frac{n}{s}$, $|\la b \ra| = \frac{n}{t}$, and $\gcd(s,t)=1$.

If $m_{ij}=m_{i'j'}$ with $0 \le i' \le i \le t-1$ and $0 \le j, j' \le \frac{n}{t} -1$, then $a(i - i') = b (j'-j)$. Therefore, $t | a(i-i')$. This means $t | (i-i')$ since $gcd(a,b) = 1$, which implies $i=i'$ because $0 \le i-i' < t$. Since $a(i - i') = b (j'-j)$ it is also the case that $j=j'$. Therefore, each element of $\mathbb{Z}_n$ will be an entry somewhere in the first $t$ rows of $M$. In fact, the first $t$ rows of $M$ are just the $t$ cosets of $\la b \ra$ in $\mathbb{Z}_n$.

Now let $M'$ be the $\frac{n}{st} \times \frac{n}{st}$ block matrix created by partitioning $M$ into $t \times s$ blocks. Let $A_{i,j}$ be the $ij$th block of $M'$. Note that $A_{i+1, j} = A_{i,j} + at $ and  $A_{i, j+1} = A_{i,j} +bs$ and $|\la at \ra| = |\la bs \ra| = \frac{n}{st}$, so the matrix $A_{i,j} + k(bs)$ appears as a block in the $i$th row of $M'$ for each integer $k$. Furthermore, $a= ps$ for some $p \in \mathbb{Z}$ and $bq \equiv t \bmod n$ for some $q \in \mathbb{Z}$ since $t = gcd(b,n)$. Therefore,  $A_{i+1, j} = A_{i,j} + (pq)bs $, so $A_{i+1,j}$ is equal to some block in the $i$th row of $M'$.

This means that the $(i+1)$st row of $M'$ is the $i$th row of $M'$ shifted by $pq$, for all $i$. So coloring each matrix in $M'$ with the same ell-tile 2-coloring from Lemma \ref{matrixcoloring} will ensure that $M$ is a well-defined ell-tile 2-coloring. It is well-defined since every element is colored and each time an element appears it recieves the same color. It is an ell-tile 2-coloring because it is periodic using an ell-tile 2-coloring that `wraps around'. This yields an $S$-polychromatic coloring of $\mathbb{Z}_n$ with two colors. 
\end{proof}
\vspace{-2mm}
Here is an example of how to get the coloring of $\mathbb{Z}_{105}$ when $a=18$ and $b=25$. The matrix $M$ is
\scriptsize
\[
\left[
\begin{array}{ccc|ccc|ccc|ccc|ccc|ccc|ccc}
0&25&50&75&100&20&45&70&95&15&40&	65&	90&	10&	35&	60&	85&	5&30&55&80\\
18&	43&	68&	93&	13&	38&	63&	88&	8&	33&	58&	83&	3&	28&	53&	78&	103&	23&	48&	73&	98\\
36&	61&	86&	6&	31&	56&	81&	1&	26&	51&	76&	101&	21&	46&	71&	96&	16&	41&	66&	91&	11\\
54&	79&	104&	24&	49&	74&	99&	19&	44&	69&	94&	14&	39&	64&	89&	9&	34&	59&	84&	4&	29\\
72&	97&	17&	42&	67&	92&	12&	37&	62&	87&	7&	32&	57&	82&	2&	27&	52&	77&	102&	22&	47\\
\hline
90&	10&	35&	60&	85&	5&	30&	55&	80&	0&	25&	50&	75&	100&	20&	45&	70&	95&	15&	40&	65\\
3&	28&	53&	78&	103&	23&	48&	73&	98&	18&	43&	68&	93&	13&	38&	63&	88&	8&	33&	58&	83\\
21&	46&	71&	96&	16&	41&	66&	91&	11&	36&	61&	86&	6&	31&	56&	81&	1&	26&	51&	76&	101\\
39&	64&	89&	9&	34&	59&	84&	4&	29&	54&	79&	104&	24&	49&	74&	99&	19&	44&	69&	94&	14\\
57&	82&	2&	27&	52&	77&	102&	22&	47&	72&	97&	17&	42&	67&	92&	12&	37&	62&	87&	7&	32\\
\hline
75&	100&	20&	45&	70&	95&	15&	40&	65&	90&	10&	35&	60&	85&	5&	30&	55&	80&	0&	25&	50\\
93&	13&	38&	63&	88&	8&	33&	58&	83&	3&	28&	53&	78&	103&	23&	48&	73&	98&	18&	43&	68\\
6&	31&	56&	81&	1&	26&	51&	76&	101&	21&	46&	71&	96&	16&	41&	66&	91&	11&	36&	61&	86\\
24&	49&	74&	99&	19&	44&	69&	94&	14&	39&	64&	89&	9&	34&	59&	84&	4&	29&	54&	79&	104\\
42&	67&	92&	12&	37&	62&	87&	7&	32&	57&	82&	2&	27&	52&	77&	102&	22&	47&	72&	97&	17\\
\hline
60&	85&	5&	30&	55&	80&	0&	25&	50&	75&	100&	20&	45&	70&	95&	15&	40&	65&	90&	10&	35\\
78&	103&	23&	48&	73&	98&	18&	43&	68&	93&	13&	38&	63&	88&	8&	33&	58&	83&	3&	28&	53\\
96&	16&	41&	66&	91&	11&	36&	61&	86&	6&	31&	56&	81&	1&	26&	51&	76&	101&	21&	46&	71\\
9&	34&	59&	84&	4&	29&	54&	79&	104&	24&	49&	74&	99&	19&	44&	69&	94&	14&	39&	64&	89\\
27&	52&	77&	102&	22&	47&	72&	97&	17&	42&	67&	92&	12&	37&	62&	87&	7&	32&	57&	82&	2\\
\hline
45&	70&	95&	15&	40&	65&	90&	10&	35&	60&	85&	5&	30&	55&	80&	0&	25&	50&	75&	100&	20\\
63&	88&	8&	33&	58&	83&	3&	28&	53&	78&	103&	23&	48&	73&	98&	18&	43&	68&	93&	13&	38\\
81&	1&	26&	51&	76&	101&	21&	46&	71&	96&	16&	41&	66&	91&	11&	36&	61&	86&	6&	31&	56\\
99&	19&	44&	69&	94&	14&	39&	64&	89&	9&	34&	59&	84&	4&	29&	54&	79&	104&	24&	49&	74\\
12&	37&	62&	87&	7&	32&	57&	82&	2&	27&	52&	77&	102&	22&	47&	72&	97&	17&	42&	67&	92\\
\hline
30&	55&	80& 0&	25&	50&	75&	100&	20&	45&	70&	95&	15&	40&	65&	90&	10&	35&	60&	85&	5\\
48&	73&	98&	18&	43&	68&	93&	13&	38&	63&	88&	8&	33&	58&	83&	3&	28&	53&	78&	103&	23\\
66&	91&	11&	36&	61&	86&	6&	31&	56&	81&	1&	26&	51&	76&	101&	21&	46&	71&	96&	16&	41\\
84&	4&	29&	54&	79&	104&	24&	49&	74&	99&	19&	44&	69&	94&	14&	39&	64&	89&	9&	34&	59\\
102&	22&	47&	72&	97&	17&	42&	67&	92&	12&	37&	62&	87&	7&	32&	57&	82&	2&	27&	52&	77\\
\hline
15&	40&	65&	90&	10&	35&	60&	85&	5&	30&	55&	80&	0&	25&	50&	75&	100&	20&	45&	70&	95\\
33&	58&	83&	3&	28&	53&	78&	103&	23&	48&	73&	98&	18&	43&	68&	93&	13&	38&	63&	88&	8\\
51&	76&	101&	21&	46&	71&	96&	16&	41&	66&	91&	11&	36&	61&	86&	6&	31&	56&	81&	1&	26\\
69&	94&	14&	39&	64&	89&	9&	34&	59&	84&	4&	29&	54&	79&	104&	24&	49&	74&	99&	19&	44\\
87&	7&32&57&82&2&	27&	52&	77&	102&22&47&72&	97&	17&	42&	67&	92&	12&	37&	62
\end{array}
\right].\]
\normalsize

Therefore, 
\[
M' = 
\begin{bmatrix}
A_{0,0} & A_{0,1} & A_{0,2} & A_{0,3} & A_{0,4} & A_{0,5} & A_{0,6} \\
A_{0,4} & A_{0,5} & A_{0,6} & A_{0,0} & A_{0,1} & A_{0,2} & A_{0,3} \\
A_{0,1} & A_{0,2} & A_{0,3} & A_{0,4} & A_{0,5} & A_{0,6} & A_{0,0} \\
A_{0,5} & A_{0,6} & A_{0,0} & A_{0,1} & A_{0,2} & A_{0,3} & A_{0,4} \\
A_{0,2} & A_{0,3} & A_{0,4} & A_{0,5} & A_{0,6} & A_{0,0} & A_{0,1} \\
A_{0,6} & A_{0,0} & A_{0,1} & A_{0,2} & A_{0,3} & A_{0,4} & A_{0,5} \\
A_{0,3} & A_{0,4} & A_{0,5} & A_{0,6} & A_{0,0} & A_{0,1} & A_{0,2}
\end{bmatrix}.
\]

Now Lemma \ref{matrixcoloring} can be used to color each $A_{i,j}$ in the following way:

\[
\chi(A_{i,j}) =
\begin{bmatrix}
R & B & B\\
B & R & B\\
R & B & R\\
B & R & B\\
B & B & R
\end{bmatrix}.
\]

Of course only the first row of $M'$ is needed to get the coloring of $\mathbb{Z}_n$.\\
Finally, we give a proof of Theorem \ref{reducedsubsetsize3}.

\begin{proof}
By Proposition \ref{reduce0abto01b} we have either Case $i$, where $S$ is equivalent to a set of the form $\{0, 1, b\}$, or Case $ii$, where $\gcd(a,n)$ and $\gcd(b,n)$ are both greater than 1.  Proposition \ref{P3iffcond3} characterizes the sets $S$ for which $p_n(S) = 3$, and Theorem \ref{theconclof01ell} shows  show that $p_n(S) = 2$ for all other sets $S$ in Case $i$, except when $n=7$ and $b=3$.  Then Lemma \ref{2poly} takes care of Case $ii$.                    
\end{proof}

\textbf{Acknowledgments.} The work of Michael Young is supported in part by the National Science Foundation through grant  $\#1719841$.

\end{document}